\documentclass{amsart}

\usepackage{amssymb}
\usepackage{graphicx}
\usepackage{MnSymbol}

\usepackage{amsthm}

\title[Number of homomorphisms]{The number of homomorphisms from the Hawaiian earring group}
\author{Samuel M. Corson}

\theoremstyle{definition}\newtheorem{theorem}{Theorem}
\theoremstyle{definition}\newtheorem*{A}{Theorem \ref{manymanyhom}}
\theoremstyle{definition}\newtheorem*{B}{Theorem \ref{subgroups}}

\theoremstyle{definition}\newtheorem{bigtheorem}{Theorem}

\numberwithin{theorem}{section}
\theoremstyle{definition}\newtheorem{corollary}[theorem]{Corollary}
\theoremstyle{definition}\newtheorem{proposition}[theorem]{Proposition}
\theoremstyle{definition}\newtheorem{definition}[theorem]{Definition}
\theoremstyle{definition}\newtheorem{question}[theorem]{Question}
\theoremstyle{definition}\newtheorem{example}{Example}
\theoremstyle{definition}
\theoremstyle{definition}
\theoremstyle{definition}\newtheorem{lemma}[theorem]{Lemma}
\theoremstyle{definition}
\theoremstyle{definition}
\theoremstyle{definition}
\theoremstyle{definition}\newtheorem{obs}[theorem]{Observation}

\newcommand{\Hom}{\operatorname{Hom}}
\newcommand{\Red}{\operatorname{Red}}
\newcommand{\slk}{\operatorname{slk}}
\newcommand{\proj}{\operatorname{proj}}

\newcommand{\W}{\mathcal{W}}
\newcommand{\HEG}{\operatorname{HEG}}
\newcommand{\HAG}{\operatorname{HAG}}
\newcommand{\Aut}{\operatorname{Aut}}
\newcommand{\HAGim}{\operatorname{HAGim}}

\begin{document}

\address{Ikerbasque- Basque Foundation for Science and Matematika Saila, UPV/EHU, Sarriena S/N, 48940, Leioa - Bizkaia, Spain}

\email{sammyc973@gmail.com}

\begin{abstract}  We show a dichotomy for groups of cardinality less than continuum.  The number of homomorphisms from the Hawaiian earring group to such a group $G$ is either the cardinality of $G$ in case $G$ is noncommutatively slender, or the number is $2^{2^{\aleph_0}}$ in case $G$ is not noncommutatively slender.  An example of a noncommutatively slender group with nontrivial divisible element is exhibited.
\end{abstract}

\maketitle

\begin{section}{Introduction}  The fundamental group of the Hawaiian earring, which we call the \emph{Hawaiian earring group} and denote $\HEG$, has become a subject of expanding use and interest in topology and group theory (see \cite{MM}, \cite{Sm}, \cite{EK}, \cite{CC1}, \cite{CC2}, \cite{FZ}, \cite{ADTW}).  Though not itself free, the structure of $\HEG$ has some analogies to that of free groups.  The fundamental group of the infinitary torus $T^{\infty}$, which is isomorphic to the product $\prod_{\omega}\mathbb{Z}$, has comparable similarities to free abelian groups while not being itself free abelian.  The group $\HEG$ is both residually and locally free, and the group $\prod_{\omega}\mathbb{Z}$ is residually and locally free abelian.  There is an easy-to-define continuous map from the Hawaiian earring to $T^{\infty}$ which induces a surjection from $\HEG$ to $\prod_{\omega}\mathbb{Z}$.  For these and other reasons, one can imagine $\HEG$ to be the unabelian version of $\prod_{\omega}\mathbb{Z}$.

These comparisons motivate an extension of certain abelian group notions to more general settings, where definitions involving $\prod_{\omega}\mathbb{Z}$ have $\HEG$ substituted.  For example, an abelian group $A$ is said to be \emph{slender} (see \cite{F}) if for every homomorphism $\phi: \prod_{\omega} \mathbb{Z} \rightarrow A$ there is a natural number $n\in \omega$ such that $\phi\circ p_n = \phi$, where $p_n: \prod_{\omega}\mathbb{Z} \rightarrow \prod_{i=0}^{n-1}\mathbb{Z}$ is the retraction to the subgroup for which only the first $n$ coordinates are possibly nonzero.  The Hawaiian earring group similarly has retraction maps $p_n$ to free subgroups $\HEG_n$, which maps correspond to the topological retractions given by mapping all points not on the largest $n$ circles to the wedge point.  Thus Eda defines a group $G$ to be \emph{noncommutatively slender} (or \emph{n-slender} for short) if for every homomorphism $\phi: \HEG \rightarrow G$ there exists $n\in \omega$ such that $\phi \circ p_n = \phi$ (see \cite{E1}).  The abelian n-slender groups are precisely the slender abelian groups \cite[Theorem 3.3]{E1}, so n-slenderness is conceptually an extension of slenderness.

Unsurprisingly, slenderness is better understood than n-slenderness.  The slender groups are completely characterized in terms of subgroups: an abelian group is slender if and only if it is torsion-free and does not contain a subgroup isomorphic to $\mathbb{Q}$, $\prod_{\omega}\mathbb{Z}$ or the $p$-adic completion of the integers for any prime $p$ (see \cite{Nu}).  By contrast, no such straightforward characterization for n-slender groups is known.  The problem of finding such a general characterization seems intractable.  For example, if $G$ is the fundamental group of a path connected, locally path connected first countable Hausdorff space which lacks a universal cover then one gets a homomorphism (induced by a continuous map) from $\HEG$ to $G$ witnessing that $G$ is not n-slender.  Such fundamental groups seem varied and complicated.

Even for countable groups it is unknown as of this writing whether a group is n-slender if and only if it is torsion-free and does not contain $\mathbb{Q}$, though certainly it is necessary that n-slender groups not contain torsion or $\mathbb{Q}$.  Many well known groups are known to be n-slender.  For example, free groups, free abelian groups, torsion-free word hyperbolic groups, torsion-free one-relator groups, and Thompson's group $F$ are n-slender (see \cite{Hi}, \cite{D}, \cite{Co1}, \cite{Co2}, \cite{CoCo}).  Despite our rather limited knowledge of small cardinality n-slender groups we get the following ($|X|$ is the cardinaliy of $X$ and $\Hom(H, G)$ denotes the set of group homomorphisms from $H$ to $G$):

\begin{bigtheorem}\label{manymanyhom}  If $G$ is a group with $|G| <2^{\aleph_0}$ then 

\begin{center}
$|\Hom(\HEG, G)| = \begin{cases}|G|$ if $G$ is n-slender$\\2^{2^{\aleph_0}}$ if $G$ is not n-slender$\end{cases}$

\end{center}

\end{bigtheorem}

The difficulty in proving Theorem \ref{manymanyhom} lies in producing $2^{2^{\aleph_0}}$-many distinct homomorphisms from $\HEG$ to $G$ given the existence of a map witnessing that $G$ is not n-slender.  One cannot simply take such a map and precompose with sufficiently many endomorphisms of $\HEG$, for $\HEG$ has only $2^{\aleph_0}$-many endomorphisms (this follows immediately from \cite[Corollary 2.11]{E2}).  Previously it was known that for any nontrivial finite group $G$ there exist $2^{2^{\aleph_0}}$-many surjections from $\HEG$ to $G$ \cite{CS}. More generally it is now known that given a compact Hausdorff topological group $G$ there are $|G|^{2^{\aleph_0}}$ homomorphisms from $\HEG$ to $G$.  This latter fact is seen by combining the main result of \cite{T} with \cite[Theorem 0.1]{Z}.  The dichotomy of Theorem \ref{manymanyhom} fails to hold when $G$ is of size $2^{\aleph_0}$ since $\HEG$ is a group which is not n-slender and $\Hom(\HEG, \HEG)$ is of cardinality $2^{\aleph_0}$.  We give the abelian version of Theorem \ref{manymanyhom} as well (see Theorem \ref{abelianhomomorphims}), though it is a great deal easier to prove.

Theorem \ref{manymanyhom} is proved using the fundamental group of the harmonic archipelago \cite{BS}, which we denote $\HAG$.  The group $\HAG$ is a quotient of $\HEG$ but enjoys much greater freedom as to mappings, as is witnessed in the following theorem:

\begin{bigtheorem}\label{subgroups}  The group $\Aut(\HAG)$ contains an isomorphic copy of the full symmetric group $S_{2^{\aleph_0}}$ on a set of cardinality continuum.  Thus $\Aut(\HAG)$ is of cardinality $2^{2^{\aleph_0}}$ and all groups of size at most $2^{\aleph_0}$ are subgroups.
\end{bigtheorem}

\noindent The group $\Aut(\HEG)$ is only of size $2^{\aleph_0}$ (using \cite[Corollary 2.11]{E2}).  

Section \ref{Prelim} provides some necessary background definitions and results.  In Section \ref{manyhom} we prove Theorems \ref{manymanyhom} and \ref{subgroups}.  In Section \ref{Residualslenderness} we turn our attention to homomorphisms \emph{from} a group to an n-slender group and consider two subgroups which completely determine whether a group of small cardinality is n-slender.  Using the machinery used to compare these subgroups we produce some more examples of n-slender groups, including an n-slender group with a nontrivial divisible element (see Example \ref{divisible}).  We then prove a theorem characterizing the n-slender subgroups of $\HEG$ via subgroups and pose two questions related to this theorem.
\end{section}

\begin{section}{Preparatory Results}\label{Prelim}

We give some necessary background information, including a combinatorial characterization of the Hawaiian earring group and the harmonic archipelago group, as well as some relevant facts related to them.  The Hawaiian earring group $\HEG$ can be described as a set of infinitary words on a countable alphabet.  Let $\{a_n^{\pm 1}\}_{n\in \omega}$ be a countably infinite set where each element has a formal inverse.  A \emph{word} $W$ is a function whose domain is a totally orderred set $\overline{W}$, whose codomain is $\{a_n^{\pm 1}\}_{n\in \omega}$ such that for each $n\in \omega$ the set $\{i\in \overline{W}\mid W(i) \in \{a_n^{\pm 1}\}\}$ is finite.  It follows that for any word $W$ the domain $\overline{W}$ is a countable total order.  We understand two words $W_0$ and $W_1$ to be the same, and write $W_0 \equiv W_1$, provided there exists an order isomorphism $\iota: \overline{W_0} \rightarrow W_1$ such that $W_1(\iota(i)) = W_0(i)$.  Let $\W$ denote the set of $\equiv$ equivalence classes.  For each $n\in \omega$ we define the function $p_n: \W \rightarrow \W$ by the restriction $p_n(W) = W\upharpoonright\{i\in \overline{W}\mid W(i) \in \{a_{m}^{\pm 1}\}_{m=0}^{n-1}\}$.  Clearly $p_m\circ p_n =p_m$ whenever $m\geq n$.  The word $p_n(W)$ has finite domain, and we write $W_0 \sim W_1$ if for every $n\in \omega$ we have $p_n(W_0)$ equal to $p_n(W_1)$ as elements in the free group over $\{a_{m}\}_{m=0}^{n-1}$.  Write $[W]$ for the equivalence class of $W$ under $\sim$.  

Given two words $W_0$ and $W_1$ we define their concatenation $W_0W_1$ to be the word whose domain is the disjoint union of $\overline{W_0}$ and $\overline{W_1}$ under the order extending that of the two subsets which places elements of $\overline{W_0}$ below those of $\overline{W_1}$, and such that $W_0W_1(i) = \begin{cases}W_0(i)$ if $i\in \overline{W_0}\\W_1(i)$ if $i\in \overline{W_1}\end{cases}$.  More generally suppose $\{W_{\lambda}\}_{\lambda \in \Lambda}$ is a collection of words with $\Lambda$ a totally ordered set and for each $n\in \omega$ the set $\{\lambda\in \Lambda\mid (\exists i\in \overline{W_{\lambda}}) W_{\lambda}(i) \in \{a_m^{\pm 1}\}_{m = 0}^{n-1}\}$ is finite.  We obtain an infinite concatenation $W \equiv \prod_{\lambda \in \Lambda}W_{\lambda}$ by setting $\overline{W}$ to be the disjoint union $\bigsqcup_{\lambda\in \Lambda}\overline{W_{\lambda}}$ under the obvious order and setting $W(i) = W_{\lambda}(i)$ where $i\in \overline{W_{\lambda}}$.  We'll use the notation $\prod_{\lambda \in \Lambda}I_{\lambda}$ for the concatenation of ordered sets $I_{\lambda}$ in the obvious way.  Given $W\in \W$ we let $W^{-1}$ be the word whose domain is $\overline{W}$ under the reverse order and such that $W^{-1}(i) = (W(i))^{-1}$.

The set $\W/\sim$ has a group structure defined by letting $[W_0][W_1] = [W_0W_1]$ and $[W]^{-1} = [W^{-1}]$.  The identity element is the $\sim$ class of the empty word $E$.  This group is isomorphis to the fundamental group of the Hawaiian earring and we denote it $\HEG$.  For each $n\in \omega$ the word map $p_n$ defines a retraction homomorphism, also denoted $p_n$, which takes $\HEG$ to a subgroup which is isomorphic to the free group on $\{a_m\}_{m=0}^{n-1}$, which we denote $\HEG_n$.  Again, $p_m\circ p_n = p_m$ whenever $m \geq n$.  The set of all elements of $\HEG$ which have a representative using no letters in $\{a_m^{\pm 1}\}_{m=0}^{n-1}$ is also a retract subgroup, which we denote $\HEG^n$.  There is a natural isomorphism $\HEG \simeq \HEG_n * \HEG^n$.

\begin{definition}  A group $G$ is \emph{noncommutatively slender}, or \emph{n-slender}, if for every homomorphism $\phi: \HEG \rightarrow G$ there exists $n\in \omega$ such that $\phi\circ p_n = \phi$.
\end{definition}

\noindent Equivalently, $G$ is n-slender if for every homomorphism $\phi: \HEG\rightarrow G$ there exists $n\in \omega$ such that $\HEG^n \leq \ker(\phi)$.  The n-slender groups do not contain torsion or $\mathbb{Q}$ as a subgroup (\cite[Theorem 3.3]{E1}, \cite{Sa}).  Free (abelian) groups, one-relator groups, and a host of other groups are n-slender (see \cite{Co2}, \cite[Theorems A, B]{CoCo}).

As is the case with a free group, there exists a characterization of $\HEG$ which utilizes so-called reduced words.  We say $W\in \W$ is \emph{reduced} if for every writing $W\equiv W_0W_1W_2$ such that $W_1\sim E$ we have $W_1 \equiv E$.  It is clear that if $W$ is reduced then so is $W^{-1}$, and if $W \equiv W_0W_1$ is reduced then so are $W_0$ and $W_1$.  We have the following (see \cite[Theorem 1.4, Corollary 1.7]{E1}):

\begin{lemma}\label{reduced}  Given $W\in \W$ there exists a reduced word $W_0 \in \W$ such that $[W] = [W_0]$, and this $W_0$ is unique up to $\equiv$.  Moreover, letting $W_0$ and $W_1$ be reduced, there exist unique words $W_{0, 0}, W_{0, 1}, W_{1, 0}, W_{1, 1}$ such that 

\begin{enumerate}  \item $W_0 \equiv W_{0, 0}W_{0, 1}$

\item $W_1 \equiv W_{1, 0}W_{1, 1}$

\item $W_{0, 1}\equiv W_{1, 0}^{-1}$

\item $W_{0, 0}W_{1, 1}$ is reduced
\end{enumerate}

\end{lemma}

\noindent  Thus one can consider $\HEG$ to be the set of all reduced words $\Red\subseteq \W$ and define the binary operation via the concatenation (4).

Some endomorphisms of $\HEG$ can be defined by simply using mappings of the set $\{a_n\}_{n\in\omega}$.  For this, suppose that $\{W_n\}_{n\in \omega}$ is a collection of words such that for every $m\in \omega$ the set $\{n\in \omega: (\exists i\in \overline{W_n}) W_n(i) \in \{a_m^{\pm 1}\}\}$ is finite.  Defining $f:\{a_n^{\pm 1}\}_{n\in\omega} \rightarrow \{W_n\}_{n\in \omega}$ by $f(a_n^{\pm 1}) = W_n^{\pm 1}$ one can extend $f$ to all of $\W$ by letting $f(W)\equiv \prod_{i\in \overline{W}}f(W(i))$.  This induces an endomorphism $\phi_f: \HEG\rightarrow \HEG$ by letting $\phi_f([W]) = [f(W)]$ (see \cite[Proposition 1.9]{E1}).  Surprisingly, every endomorphism is equal, up to conjugation, to a homomorphism defined in this way (this is essentially \cite[Corollary 2.11]{E2}):

\begin{lemma}\label{continuoustoconjugation}  If $\phi\in \Hom(\HEG, \HEG)$ there exists $h\in \HEG$ and a mapping $f: a_n\mapsto W_n$ such that $\phi(g) = h^{-1}\phi_f(g)h$.
\end{lemma}

An important quotient of $\HEG$ is the so called harmonic archipelago group $\HAG$, which we define by $\HEG/\langle\langle\bigcup_{n\in \omega}\HEG_n\rangle\rangle$.  This uncountable group is isomorphic to the fundamental group of the harmonic archipelago (see \cite{CHM} or \cite{Ho}).  Let $\pi:\HEG \rightarrow \HAG$ denote the quotient map.  We write $[[W_0]] = [[W_1]]$ if $\pi([W_0]) = \pi([W_1])$.  Since $[[W]]=[[E]]$ for any finite $W\in \W$, it is clear that assuming $W \equiv W_0W_1W_2$ with $W_1$ finite, we have $[[W]]=[[W_0W_2]]$.  In consequence, we obtain:

\begin{lemma}\label{hagfact1}  For all $n\in \omega$, we have $\pi(\HEG^n) = \HAG$.  If $\HAG$ has a nontrivial homomorphic image in $G$ then $G$ is not n-slender.
\end{lemma}

\begin{proof}  The first claim follows from the fact that $\HEG \simeq \HEG_n *\HEG^n$ and elements of $\HEG_n$ have a representative which is a finite word.  For the second claim, if $\phi(\HAG)$ has nontrivial image then $\phi \circ \pi$ witnesses that $G$ is not n-slender by the first claim.
\end{proof}

When $|G|<2^{\aleph_0}$, the converse of the second claim in Lemma \ref{hagfact1} holds as well.  First we state a lemma (which is a special instance of \cite[Theorem 4.4 (1)]{CC2}):

\begin{lemma}\label{stabilize}  If $|G|< 2^{\aleph_0}$ and $\phi:\HEG \rightarrow G$ is a homomorphism then the sequence of images $\phi(\HEG^n)$ eventually stabilizes.
\end{lemma}

The following is an argument of Conner \cite{C}:

\begin{lemma}\label{persistentimage}  Suppose $\phi: \HEG \rightarrow G$ and $S \subseteq \bigcap_{n\in \omega}\phi(\HEG^n)$ with $S$ countable.  Then there exists a homomorphism $\psi: \HAG \rightarrow G$ with $\psi(\HAG) \supseteq S$.
\end{lemma}

\begin{proof}  Let $S = \{g_0, g_1, \ldots\}$ be an enumeration of $S$.  For each $m\in \omega$ and $n\geq m$ select $[W_{(m, n)}]\in \HEG^n$ such that $\phi([W_{(m, n)}]) = g_m$.  Let $f: \omega \rightarrow \{(m, n)\in \omega^2: n \geq m\}$ be a bijection.  We get an endomorphism $\beth: \HEG \rightarrow \HEG$ such that $\beth([a_k]) = [W_{f(k)}]$, and letting $\phi_0 = \phi \circ \beth$ we see that for each $m\in \omega$ there exist arbitrarily large $k\in \omega$ such that $\phi_0([a_k]) = g_m$.

Now for each $m\in \omega$ we let $N_m = \{k\in \omega: \phi_0([a_k]) = g_m\}$  and notice that $\omega = \bigsqcup_{m\in \omega}N_m$ and each $N_m$ is infinite.  Enumerate each $N_m$ in the standard way $N_m = \{k_{(m, 0)}, k_{(m, 1)}, \ldots\}$ so that $k_{(m, i)} < k_{(m, i+1)}$.  Let $\gamma: \omega\rightarrow \omega^2$ be a bijection.  Define the endomorphism $\tau:\HEG \rightarrow \HEG$ such that $\tau([a_p]) = [a_{\gamma(p)}a_{\proj_1(\gamma(p)), \proj_2(\gamma(p))+1}^{-1}]$ (here $\proj_i$ denotes projection to the $i$ coordinate).  Notice now that $\phi_1 = \phi_0 \circ \tau:\HEG \rightarrow \HEG$ has $\phi_1([a_j]) = 1_G$ for every $j\in \omega$.  Thus $\phi_1$ descends to a homomorphism $\psi: \HAG \rightarrow G$ with the same image as $\phi_1$.  We shall be done if we show that $S \subseteq \phi(\HEG) = \psi(\HAG)$.

Letting $m\in \omega$ be given we notice that 

\begin{center} $g_m = \phi_0([a_{k_{(m, 0)}}])$

$= \phi_0([a_{k_{(m, 0)}}a_{k_{(m, 1)}}^{-1}a_{k_{(m, 1)}}a_{k_{(m, 2)}}^{-1}a_{k_{(m, 2)}}a_{k_{(m, 3)}}^{-1}a_{k_{(m, 3)}}a_{k_{(m, 4)}}^{-1}\cdots])$

$= \phi_1 \circ \tau([a_{\gamma^{-1}(m, 0)}a_{\gamma^{-1}(m, 1)}\cdots])$

$= \psi([[a_{\gamma^{-1}(m, 0)}a_{\gamma^{-1}(m, 1)}\cdots]]) \in \psi(\HAG)$
\end{center}
\end{proof}

Combining Lemmas \ref{stabilize} and \ref{persistentimage} one immediately obtains:

\begin{lemma}\label{hagfact2}  If $|G|<2^{\aleph_0}$ with $G$ not n-slender, then there is a nontrivial homomorphism from $\HAG$ to $G$.
\end{lemma}

Next we state a special case of \cite[Theorem 1.3]{E3}:

\begin{lemma}\label{funneling}  If $\phi: \HEG \rightarrow *_{j\in J}H_j$ is a homomorphism to a free product, then for some $n\in \omega$ and $j\in J$ the image $\phi(\HEG^n)$ lies in a conjugate of $H_j$.
\end{lemma}

\noindent Consequently we obtain:

\begin{lemma}\label{hagfact3}  If $\phi: \HAG \rightarrow *_{j\in J}H_j$ then for some $j\in J$ the image $\phi(\HAG)$ is a subgroup of a conjugate of $H_j$.
\end{lemma}

Finally, we note that extending the set $\{a_n^{\pm 1}\}_{n\in \omega}$ to $\{a_n^{\pm 1}, b_n^{\pm 1}, c_n^{\pm 1}\}_{n\in \omega}$, we analogously define the set $\W_{a, b, c}$ of words, the set of reduced words $\Red_{a, b, c}\subseteq \W_{a, b, c}$, the Hawaiian earring group $\HEG_{a, b, c}$ and the harmonic archipelago group $\HAG_{a, b, c}$.  Using the bijection $\Gamma: \{a_n^{\pm 1}\}_{n\in \omega}\rightarrow \{a_n^{\pm 1}, b_n^{\pm 1}, c_n^{\pm 1}\}_{n\in \omega}$ given by

\begin{center}
$\Gamma(a_n^{\pm 1}) \mapsto \begin{cases} a_{m}^{\pm 1}$ if $n = 3m\\b_m^{\pm 1}$ if $n = 3m +1\\ c_n^{\pm 1}$ if $n = 3m +2\end{cases}$
\end{center}

\noindent the elements of $\W$ are placed in bijection with those of $\W_{a, b, c}$ by letting $\Gamma(W) = \prod_{i\in \overline{W}}\Gamma(W(i))$.  This bijection satisfies $\Gamma(\Red) = \Red_{a, b, c}$ and induces an isomorphism $\HEG \simeq \HEG_{a, b, c}$ where the latter group is defined analogously to the former.  This isomorphism also descends to an isomorphism $\HAG \simeq \HAG_{a, b, c}$.  By deleting the elements of $\{b_n^{\pm 1}, c_n^{\pm1}\}_{n\in \omega}$ we obtain retractions $r_a: \W_{a, b, c} \rightarrow \W$, $r_a: \HEG_{a, b, c} \rightarrow \HEG$, and $r_a: \HAG_{a, b, c} \rightarrow \HAG$.    We can similarly define the set of words $\W_{b, c} \subseteq \W_{a, b, c}$ which have range disjoint from $\{a_n^{\pm 1}\}_{n\in \omega}$, the set of reduced words $\Red_{b, c} = \W_{b, c}\cap \Red_{a, b, c}$, an isomorph of the Hawaiian earring group $\HEG_{b, c} \leq \HEG_{a, b, c}$ and an isomorph of the harmonic archipelago group $\HAG_{b, c}\leq \HAG_{a, b, c}$.  The inclusion $\HAG_{b, c} \leq \ker(r_a)$ holds.
\end{section}

\begin{section}{Theorems \ref{manymanyhom} and \ref{subgroups}}\label{manyhom}
We begin by stating and proving the abelian version of Theorem \ref{manymanyhom}.

\begin{theorem}\label{abelianhomomorphims}  If $A$ is an abelian group with $|A| <2^{\aleph_0}$ then 

\begin{center}
$|\Hom(\prod_{\omega}\mathbb{Z}, A)| = \begin{cases}|A|$ if $A$ is slender$\\2^{2^{\aleph_0}}$ if $A$ is not slender$\end{cases}$

\end{center}

\end{theorem}

\begin{proof}  Suppose that $A$ is slender.  If $A$ is the trivial group then there is exactly one homomorphism from $\prod_{\omega}\mathbb{Z}$ to $A$.  If $A$ is nontrivial then $A$ is infinite, torsion-free.  Since $A$ is slender we see that $\Hom(\prod_{\omega}\mathbb{Z}, A) = \bigcup_{n\in \omega} \{\phi:\prod_{\omega} \mathbb{Z} \rightarrow A\mid \phi\circ p_n = \phi\}$.  The set $\{\phi:\prod_{\omega} \mathbb{Z} \rightarrow A\mid \phi\circ p_1 = \phi\}$ has cardinality exactly $|A|$, and since $A$ is infinite it is in fact true that for each $n\in \omega$ we have $\{\phi:\prod_{\omega} \mathbb{Z} \rightarrow A\mid \phi\circ p_n = \phi\}$ is of cardinality $|A|$.  Then $\Hom(\prod_{\omega}\mathbb{Z}, A)$ has cardinality $|A|$ in this case as well.

Suppose that $A$ is not slender.  By \cite{Sa} we know that $A$ must contain an isomorphic copy of $\mathbb{Q}$ or of the cyclic group $\mathbb{Z}/p$ for some prime $p$.  If $\mathbb{Q} \leq A$, we take a subgroup $F\leq \prod_{\omega}\mathbb{Z}$ which is a free abelian group of rank $2^{\aleph_0}$.  The construction of such an $F$ follows a straightforward induction.  There are $2^{2^{\aleph_0}}$ distinct homomorphisms from $F$ to $\mathbb{Q}$, and since $\mathbb{Q}$ is an injective $\mathbb{Z}$-module, each of these homomorphisms may be extended to $\prod_{\omega}\mathbb{Z}$.  Thus in this case there are at least $2^{2^{\aleph_0}}$ homomorphisms from $\prod_{\omega}\mathbb{Z}$ to $A$, and since $|A| <2^{\aleph_0}$ we have precisely $2^{2^{\aleph_0}}$ homomorphisms.

If $\mathbb{Z}/p \leq A$ then we use the epimorphism $\epsilon:\prod_{\omega}\mathbb{Z}\rightarrow \prod_{\omega}(\mathbb{Z}/p)$ and notice that $\prod_{\omega}(\mathbb{Z}/p)$ is a vector space over the field $\mathbb{Z}/p$, so by selecting a basis we obtain a group isomorphism $\prod_{\omega}(\mathbb{Z}/p)\simeq \bigoplus_{2^{\aleph_0}}(\mathbb{Z}/p)$.  For each $\mathcal{S} \subseteq 2^{\aleph_0}$ we get a homomorphism $\epsilon_{\mathcal{S}}: \bigoplus_{2^{\aleph_0}}(\mathbb{Z}/p) \rightarrow \mathbb{Z}/p$ given by taking the sum of the $\mathcal{S}$ coordinates.  Each such $\epsilon_{\mathcal{S}}$ is a distinct homomorphism and so each composition $\epsilon_{\mathcal{S}} \circ \epsilon: \prod_{\omega}\mathbb{Z}\rightarrow \mathbb{Z}/p$ is distinct.  Thus there exist at least $2^{2^{\aleph_0}}$ homomorphisms from  $\prod_{\omega}\mathbb{Z}$ to $A$, and again by $|A| < 2^{\aleph_0}$ we see that there are exactly $2^{2^{\aleph_0}}$ homomorphisms.
\end{proof}

Next, we prove Theorem \ref{manymanyhom} modulo a proposition which proves the existence of many homomorphisms.

\begin{A}  If $G$ is a group with $|G| <2^{\aleph_0}$ then 

\begin{center}
$|\Hom(\HEG, G)| = \begin{cases}|G|$ if $G$ is n-slender$\\2^{2^{\aleph_0}}$ if $G$ is not n-slender$\end{cases}$

\end{center}

\end{A}

\begin{proof}   Suppose $G$ is n-slender and $|G|<2^{\aleph_0}$.  If $G$ is trivial then there is exactly one homomorphism from $\HEG$ to $G$.  If $G$ is nontrivial then $G$ is infinite and $\Hom(\HEG, G) = \bigcup_{n\in\omega}\{\phi: \HEG\rightarrow G\mid \phi\circ p_n = \phi\}$.  Moreover since $\HEG \simeq \HEG_n *\HEG^n$ for all $n\in\omega$ and $\HEG_n$ is a free group of rank $n$, we get $\Hom(\HEG, G)$ as a countable union of sets of cardinality $|G|$.  Thus $|\Hom(\HEG, G)| = |G|$ in either case.

Suppose $G$ is not n-slender and $|G|<2^{\aleph_0}$.  By Lemma \ref{hagfact2} there exists a nontrivial homomorphism from $\HAG$ to $G$.  By Proposition \ref{inforapenny} there are at least $2^{2^{\aleph_0}}$ homomorphisms from $\HAG$ to $G$ and precomposing these homomorphisms with the surjective map $\pi: \HEG \rightarrow \HAG$ we obtain at least $2^{2^{\aleph_0}}$ homomorphisms from $\HEG$ to $G$.  Since $|\HEG| = 2^{\aleph_0}$ and $|G|<2^{\aleph_0}$ we get $|\Hom(\HEG, G)| = 2^{2^{\aleph_0}}$.
\end{proof}

\begin{proposition}\label{inforapenny}  If $|\Hom(\HAG, G)| >1$ then $|\Hom(\HAG, G)| \geq 2^{2^{\aleph_0}}$.
\end{proposition}

We prove Proposition \ref{inforapenny} after a sequence of lemmas.

\begin{lemma}\label{nicehom}  If $\phi_0: \HAG \rightarrow G$ is a nontrivial homomorphism there exists a homomorphism $\phi: \HAG_{a, b, c} \rightarrow G$ such that $\phi([[a_0a_1a_2\cdots]])\neq 1_G$ and $\HAG_{b, c}\leq \ker(\phi)$.
\end{lemma}

\begin{proof}  Select $g\in \phi_0(\HAG) \setminus \{1_G\}$.  By Lemma \ref{hagfact1} we select for each $n\in \omega$ an element $[W_n]\in \HEG^n$ such that $\phi_0\circ \pi([W_n]) = g$.  Let $\sigma_0: \HEG \rightarrow \HEG$ be the endomorphism determined by $a_n \mapsto W_n$.  Letting $\phi_1 = \phi_0\circ \sigma_0:\HEG \rightarrow G$, we have $\phi_1([a_n]) = g$ for all $n\in \omega$.

Now let $\sigma_1: \HEG \rightarrow \HEG$ be the endomorphism determined by $a_n \mapsto a_{n}a_{n+1}^{-1}$.  Notice that $\phi_2 = \phi_1\circ \sigma_1$ satisfies $\phi_2([a_n]) = 1_G$ for all $n\in \omega$.  Thus $\phi_2$ descends to a map $\phi': \HAG\rightarrow G$ by letting $\phi'([[W]]) = \phi_2([W])$.  Moreover we have $\phi_2([a_0a_1a_2\cdots ]) = g$ and so $\phi'([[a_0a_1\cdots]]) =g$.  Letting $r_a:\HAG_{a, b, c}\rightarrow \HAG$ be the retraction map, define $\phi: \HAG_{a, b, c} \rightarrow G$ by $\phi'\circ r_a$.  Thus $\HAG_{b, c}\leq \ker(\phi)$ and $\phi([[a_0a_1\cdots]])\neq 1_G$.
\end{proof}

Let $\Sigma$ be a collection of subsets of $\omega$ such that each $S\in \Sigma$ is infinite, for distinct $S_0, S_1\in \Sigma$ the intersection $S_0\cap S_1$ is finite, and $|\Sigma| = 2^{\aleph_0}$.  Such a construction is fairly straightforward (see  \cite[II.1.3]{Ku}).  For each $S\in \Sigma$ define a word $U_S\in \Red_{b, c}$ by $\overline{U_S} = \omega$ and $U_S(n) = \begin{cases}b_n$ if $n\in S\\ c_n$ if $n\notin S  \end{cases}$.  For each $n\in \omega$ and $S\in \Sigma$ let $U_{S, n} = W\upharpoonright (\omega \setminus \{0, \ldots, n-1\})$.  Thus $U_{S, 0} = U_S$ and for each $n\in \omega$ we have $[U_{S, n}]\in \HEG_{b, c}^n$.  Let $T$ be a symbol such that $T\notin \Sigma$ and define a word $U_T\in \Red$ by $U_T \equiv a_0a_1a_2\cdots$ and let $U_{T, n} \equiv a_na_{n+1}\cdots$.

Given a word $W\in \Red_{a, b, c}$ we say an interval $I\subseteq \overline{W}$ \emph{participates in $\Sigma$ for $W$} if  $W\upharpoonright I \equiv U_{S, n}$ or $W\upharpoonright I \equiv U_{S, n}^{-1}$ for some $S\in \Sigma$ and $n\in \omega$.  Similarly, given a word $W\in \Red_{a, b, c}$ we say an interval $I\subseteq \overline{W}$ is  \emph{maximal in $\Sigma$ for $W$} if $I$ participates in $\Sigma$ and there does not exist a strictly larger interval $\overline{W}\supset I' \supset I$ which participates in $\Sigma$.

\begin{lemma}\label{existenceofmaximal}  If $W\in \Red_{a, b, c}$ and $I \subseteq \overline{W}$ is an interval which participates in $\Sigma$ for $W$ then $I$ is contained in a unique interval $I' \supseteq I$ which is maximal in $\Sigma$ for $W$.
\end{lemma}

\begin{proof}  Let $I \subseteq \overline{W}$ satisfy the hypotheses.  Suppose $W\upharpoonright I \equiv U_{S, n}$.  By the definition of the $U_{S,n}$, it is clear that if $m\in \omega$ and $S' \in \Sigma$ also satisfy $W\upharpoonright I \equiv U_{S', m}$ then $S = S'$ and $m = n$.  Also, it cannot be that $W \upharpoonright  I \equiv U_{S', m}^{-1}$, since $\overline{U_{S, n}}$ has order type $\omega$ and $\overline{U_{S', m}^{-1}}$ has order type $-\omega$.  If there does not exist an immediate predecessor $i < \min(I)$ such that $W\upharpoonright (I \cup \{i\}) \equiv U_{S, n-1}$ then $I$ is maximal in $\Sigma$ for $W$.  Otherwise we get $W\upharpoonright (I \cup \{i\}) \equiv U_{S, n-1}$ and apply induction on $n$.

The proof in case $W\upharpoonright I \equiv U_{S, n}^{-1}$ is similar.
\end{proof}

\begin{lemma}\label{disjoint}  Given $W\in \Red_{a, b, c}$, if subintervals $I_0, I_1 \subseteq \overline{W}$ are both maximal in $\Sigma$ for $W$ then either $I_0\cap I_1 = \emptyset$ or $I_0 = I_1$.
\end{lemma}

\begin{proof}  Suppose $I_0 \cap I_1 \neq \emptyset$ and select $i\in I_0 \cap I_1$.  If $W(i)$ does not have superscript $-1$ then by how we have defined the words $U_{S, n}$ we see that $I_0$ and $I_1$ are both of order type $\omega$.  Since $I_0\cap I_1 \neq \emptyset$ and both $I_0$ and $I_1$ are intervals of order type $\omega$, we get either $I_0 \subseteq I_1$ or $I_1\subseteq I_0$ and since both are maximal in $\Sigma$  for $W$ we get $I_0 = I_1$ by Lemma \ref{existenceofmaximal}.  The case where $W(i)$ has superscript $-1$ is handled similarly.
\end{proof}

\begin{lemma}\label{decomposition}  If $W\in \Red_{a, b, c}$ there is a unique decomposition $\overline{W} = \prod_{l\in \Lambda} I_{\lambda}$ such that $I_{\lambda}$ is either maximal in $\Sigma$ for $W$ or a maximal interval which does not intersect with any interval which is maximal in $\Sigma$ for $W$.
\end{lemma}

\begin{proof}  To begin we let $\{I_{\lambda}\}_{\lambda \in \Lambda'}$ be the collection of intervals which are maximal in $\Sigma$ for $W$.  Next, by Zorn's Lemma we select all maximal intervals $\{I_{\lambda}\}_{\lambda \in \Lambda''}$ which are disjoint from the elements of $\{I_{\lambda}\}_{\lambda \in \Lambda'}$.  Taking $\Lambda = \Lambda' \cup \Lambda''$ and endowing this set with the obvious ordering, we see that $\overline{W} \equiv \prod_{\lambda\in \Lambda} I_{\lambda}$.  Uniqueness is clear.
\end{proof}

Lemma \ref{decomposition} gives a word $W\in \Red_{a, b, c}$ a unique decomposition $W \equiv \prod_{\lambda \in \Lambda} W\upharpoonright I_{\lambda}$.  Now, given a function $f: \Sigma \rightarrow \Sigma\cup \{T\}$ we define a function $F_f: \Red_{a, b, c} \rightarrow \W_{a, b, c}$ by letting

\begin{center}
 $F_f(W) \equiv \prod_{\lambda \in \Lambda} W_{\lambda}'$

\end{center}
 where

\begin{center}
$W \equiv \prod_{\lambda \in \Lambda}W_{\lambda}$
\end{center}

\noindent is the aforementioned decomposition implied by Lemma \ref{decomposition} and

\begin{center}
$W_{\lambda}' \equiv\begin{cases}W_{\lambda}\upharpoonright I_{\lambda}$ if $I_{\lambda}$ is not maximal in $\Sigma$ for $W\\ U_{f(S), n}$ if $W\upharpoonright I_{\lambda}\equiv U_{S, n}$ with $S\in \Sigma\\ U_{f(S), n}^{-1}$ if $W\upharpoonright I_{\lambda}\equiv U_{S, n}^{-1}$ with $S\in \Sigma\end{cases}$
\end{center}

The object $\prod_{\lambda \in \Lambda} W_{\lambda}'$ is evidently a function whose domain is a totally ordered set which is order isomorphic to $\overline{W}$.  It has as codomain $\{a_n^{\pm 1}, b_n^{\pm 1}, c_n^{\pm 1}\}_{n\in \omega}$.  Moreover, for each $n\in \omega$ the set of elements $i\in \overline{\prod_{\lambda \in \Lambda} W_{\lambda}'}$ for which the subscript of $\overline{\prod_{\lambda \in \Lambda} W_{\lambda}'}(i)$ is $\leq n$ has the same cardinality as the set of elements $i\in \overline{W}$ such that $W(i)$ has subscript $\leq n$.  Thus $\prod_{\lambda \in \Lambda} W_{\lambda}'\in \W_{a, b, c}$.

We check that the map $\psi_f:\HEG \rightarrow \HAG$ given by $\psi_{f}(W) = [[F_{f}(W)]]$ is a homomorphism.  Towards this we give the following lemma.

\begin{lemma}\label{almosthom}  If $W\in \Red_{a, b, c}$ and $W \equiv W_0W_1$ then $\psi_{f}(W) = \psi_{f}(W_0)\psi_{f}(W_1)$.
\end{lemma}

\begin{proof}  If $W\equiv\prod_{\lambda \in \Lambda} W_{\lambda}$ is the decomposition given by Lemma \ref{decomposition} then one of the following holds:

\begin{enumerate}\item  There exist $\Lambda_0, \Lambda_1 \subseteq \Lambda$ such that all elements of $\Lambda_0$ are below those of $\Lambda_1$, $W_0 \equiv \prod_{\lambda \in \Lambda_0}W_{\lambda}$ and $W_1 \equiv \prod_{\lambda \in \Lambda_1}W_{\lambda}$.

\item  There exists $\zeta \in \Lambda$ such that $W_{\zeta} \equiv W_{0, \zeta}W_{1, \zeta}$ and $W_0 \equiv (\prod_{\lambda<\zeta}W_{\lambda})W_{0, \zeta}$ and $W_1 \equiv W_{1, \zeta}(\prod_{\lambda<\zeta}W_{\lambda})$.

\end{enumerate}

In case (1) the decompositions of $W_0$ and $W_1$ given by Lemma \ref{decomposition} are respectively $W_0 \equiv \prod_{\lambda \in \Lambda_0}W_{\lambda}$ and $W_1 \equiv \prod_{\lambda \in \Lambda_1}W_{\lambda}$.  Thus in this case we get $F_f(W) \equiv F_f(W_0)F_f(W_1)$ and $\psi_f(W) = \psi_f(W_0)\psi_f(W_1)$ is immediate.

In case (2) there are several subcases.  We mention each of these subcases and state the Lemma \ref{decomposition} decomposition for $W_0$ and $W_1$.  

\noindent \textbf{2.1}  If $\overline{W_{\zeta}}$ was maximal in $\Sigma$ for $W$ with $W_{\zeta} \equiv U_{S, n}$ then $W_{1, \zeta} \equiv U_{S, n'}$ for some $n' >n$ and $W_{0, \zeta}$ will be a finite word which is a prefix to $U_{S, n}$.

\noindent\textbf{2.1.1}  If in addition to 2.1 there is an immediate predecessor $\zeta'<\zeta$ in $\Lambda$ and $\overline{W_{\zeta'}}$ is not maximal in $\Sigma$ for $W$, then the decomposition of $W_0$ is $W_0 \equiv (\prod_{\lambda < \zeta'}W_{\lambda})(W_{\zeta'}W_{0, \zeta})$ and the decomposition of $W_1$ is $W_1 \equiv W_{1, \zeta}\prod_{\zeta< \lambda}W_{\lambda}$.  By this we mean that the decomposition of word $W_0$ has index of order type $\{\lambda \in \Lambda\mid \lambda \leq \zeta'\}$ and the last word of the decomposition is $\equiv W_{\zeta'}W_{0, \zeta}$.  The decomposition of $W_1$ has index of order type $\{\lambda \in \Lambda\mid \zeta \leq \lambda\}$ and the first word of the decomposition is $\equiv W_{1, \zeta}$.  Here we get $F_f(W_0) \equiv  (\prod_{\lambda < \zeta'}W_{\lambda}')(W_{\zeta'}W_{0, \zeta})$ and $F_f(W_1) \equiv W_{1, \zeta}'\prod_{\zeta< \lambda}W_{\lambda}'$ and since the $[[\cdot]]$ class of a word is closed under modifying a finite subword we get

\begin{center}  $\psi_f(W_0)\psi_f(W_1) = [[ (\prod_{\lambda < \zeta'}W_{\lambda}')(W_{\zeta'}W_{0, \zeta})]][[W_{1, \zeta}'\prod_{\zeta< \lambda}W_{\lambda}']]$

$= [[ (\prod_{\lambda < \zeta'}W_{\lambda}')(W_{\zeta'}W_{0, \zeta})W_{1, \zeta}'\prod_{\zeta< \lambda}W_{\lambda}']]$

$= [[ (\prod_{\lambda < \zeta'}W_{\lambda}')W_{\zeta'}W_{0, \zeta}W_{1, \zeta}'\prod_{\zeta< \lambda}W_{\lambda}']]$

$= [[ (\prod_{\lambda < \zeta'}W_{\lambda}')W_{\zeta'}'W_{0, \zeta}W_{1, \zeta}'\prod_{\zeta< \lambda}W_{\lambda}']]$

$= [[ (\prod_{\lambda \leq \zeta'}W_{\lambda}')W_{0, \zeta}W_{1, \zeta}'\prod_{\zeta< \lambda}W_{\lambda}']]$

$= [[ (\prod_{\lambda <\zeta}W_{\lambda}')W_{\zeta}'\prod_{\zeta< \lambda}W_{\lambda}']]$

$= [[\prod_{\lambda \in \Lambda} W_{\lambda}']]= \psi_f(W)$

\end{center}

\noindent\textbf{2.1.2}  If in addition to 2.1 there is an immediate predecessor $\zeta'<\zeta$ in $\Lambda$ and $\overline{W_{\zeta'}}$ is maximal in $\Sigma$ for $W$ then the decomposition of $W_0$ is $W_0 \equiv  (\prod_{\lambda < \zeta'}W_{\lambda})(W_{\zeta'})(W_{0, \zeta})$ and $W_1 \equiv W_{1, \zeta}\prod_{\zeta< \lambda}W_{\lambda}$.  Here we get

\begin{center}
$\psi_f(W_0)\psi_f(W_1) = [[ (\prod_{\lambda < \zeta'}W_{\lambda}')(W_{\zeta'}')(W_{0, \zeta}')]][[W_{1, \zeta}'\prod_{\zeta< \lambda}W_{\lambda}']]$

$=[[(\prod_{\lambda < \zeta}W_{\lambda}')W_{0, \zeta}'W_{1, \zeta}'(\prod_{\zeta< \lambda}W_{\lambda}')]]$

$=[[(\prod_{\lambda < \zeta}W_{\lambda}')W_{\zeta}'(\prod_{\zeta< \lambda}W_{\lambda}')]]$

$=\psi_f(W)$
\end{center}

\noindent\textbf{2.1.3}  If there is no immediate predecessor $\zeta'<\zeta$ in $\Lambda$ then $W_0 \equiv  (\prod_{\lambda < \zeta}W_{\lambda})(W_{0, \zeta})$ and $W_1 \equiv W_{1, \zeta}\prod_{\zeta< \lambda}W_{\lambda}$ are the decompositions.  We get that

\begin{center}$\psi_f(W_0)\psi_f(W_1) = [[ (\prod_{\lambda < \zeta}W_{\lambda}')(W_{0, \zeta}')W_{1, \zeta}'\prod_{\zeta< \lambda}W_{\lambda}']]$

$=[[(\prod_{\lambda < \zeta}W_{\lambda}')W_{1, \zeta}'\prod_{\zeta< \lambda}W_{\lambda}']]$

$=[[(\prod_{\lambda < \zeta}W_{\lambda}')W_{\zeta}'\prod_{\zeta< \lambda}W_{\lambda}']]$

$=\psi_f(W)$

\end{center}

\noindent\textbf{2.2}  If $\overline{W_{\zeta}}$ was maximal in $\Sigma$ for $W$ with $W_{\zeta} \equiv U_{S, n}^{-1}$ then $W_{0, \zeta} \equiv U_{S, n'}^{-1}$ for some $n' >n$ and $W_{1, \zeta}$ will be a finite word which is a suffix to $U_{S, n}^{-1}$.

\noindent\textbf{2.2.1}  If in addition to 2.2 there is an immediate successor $\zeta< \zeta'$ in $\Lambda$ and $\overline{W_{\zeta'}}$ is not maximal in $\Sigma$ for $W$, then the decomposition of $W_0$ is $W_0 \equiv \prod_{\lambda<\zeta} W_{\lambda}(W_{0, \zeta})$ and the decomposition of $W_1$ is $W_1 \equiv (W_{1, \zeta}W_{\zeta'})\prod_{\zeta'<\lambda}W_{\lambda}$.  The claim in this subcase follows along the same lines as 2.2.1.

\noindent\textbf{2.2.2}  If in addition to 2.2 there is an immediate successor $\zeta<\zeta'$ in $\Lambda$ and $\overline{W_{\zeta'}}$ is maximal in $\Sigma$ for $W$ then the decomposition of $W_0$ is $W_0 \equiv \prod_{\lambda<\zeta} W_{\lambda}(W_{0, \zeta})$ and the decomposition of $W_1$ is $W_1 \equiv (W_{1, \zeta})(W_{\zeta'})\prod_{\zeta'<\lambda}W_{\lambda}$.  The claim in this subcase follows along the same lines as 2.1.2.

\noindent\textbf{2.2.3}  If in addition to 2.2 there is no immediate successor $\zeta <\zeta'$ in $\Lambda$ then $W_0 \equiv  (\prod_{\lambda < \zeta}W_{\lambda})(W_{0, \zeta})$ and $W_1 \equiv W_{1, \zeta}\prod_{\zeta< \lambda}W_{\lambda}$ are the decompositions.  The claim follows in this subcase along the same lines as 2.1.3.

\noindent\textbf{2.3}  If $\overline{W_{\zeta}}$ was not maximal for $\Sigma$ in $W$ then the decompostions of $W_0$ and $W_1$ are $W_0 \equiv (\prod_{\lambda<\zeta}W_{\lambda})W_{0, \zeta}$ and $W_1 \equiv W_{1, \zeta}\prod_{\zeta<\lambda}W_{\lambda}$.  Here we get

\begin{center}  $\psi_f(W_0)\psi_f(W_1) = [[ (\prod_{\lambda<\zeta}W_{\lambda}')W_{0, \zeta}']][[W_{1, \zeta}'\prod_{\zeta<\lambda}W_{\lambda}']]$

$= [[(\prod_{\lambda<\zeta}W_{\lambda}')W_{0, \zeta}W_{1, \zeta}\prod_{\zeta<\lambda}W_{\lambda}']]$

$= [[(\prod_{\lambda<\zeta}W_{\lambda}')(W_{0, \zeta}W_{1, \zeta})\prod_{\zeta<\lambda}W_{\lambda}']]$

$= [[(\prod_{\lambda<\zeta}W_{\lambda}')W_{\zeta}\prod_{\zeta<\lambda}W_{\lambda}']]$

$=  [[(\prod_{\lambda<\zeta}W_{\lambda}')W_{\zeta}'\prod_{\zeta<\lambda}W_{\lambda}']]$

$= [[\prod_{\lambda\in \Lambda}W_{\lambda}']]$

$= \psi_f(W)$
\end{center}

\noindent and the claim follows in all circumstances.

\end{proof}

\begin{lemma}\label{homomorphism}  The function $\psi_f:\HEG_{a, b, c} \rightarrow \HAG_{a, b, c}$ is a homomorphism which descends to a homomorphism $\phi_f: \HAG_{a, b, c} \rightarrow \HAG_{a, b, c}$.
\end{lemma}

\begin{proof}  Let $W_0, W_1\in \Red_{a, b, c}$.  We let

\begin{center}  $W_0 \equiv W_{0, 0}W_{0,1}$

$W_1 \equiv W_{1, 0}W_{1, 1}$

$W_{0, 1} \equiv W_{1, 0}^{-1}$

$W_{0, 0} W_{1, 1}\in \Red_{a, b, c}$
\end{center}

\noindent  according the conclusion of Lemma \ref{reduced}.  We have 

\begin{center}  $\psi_f(W_0)\psi_f(W_1) = \psi_f(W_{0, 0})\psi_f(W_{0, 1})\psi_f(W_{1, 0})\psi_f(W_{1, 1})$

$= \psi_f(W_{0, 0})\psi_f(W_{0, 1})\psi_f(W_{0, 1}^{-1})\psi_f(W_{1, 1})$

$=\psi_f(W_{0, 0})\psi_f(W_{0, 1})(\psi_f(W_{0, 1}))^{-1}\psi_f(W_{1, 1})$

$= \psi_f(W_{0, 0})\psi_f(W_{1, 1})$

$= \psi_f(W_{0, 0}W_{1, 1})$

\end{center}

\noindent where the first and last equality come from Lemma \ref{almosthom} and the third equality is clear since the map $\psi_f$ obviously satisfies $\psi_f(W^{-1}) = \psi_f(W)^{-1}$.  Thus $\psi_f$ is a homomorphism and it is easy to check that $\bigcup_{n\in \omega}\HEG_{n, a, b, c}\leq \ker(\psi_f)$, and so $\psi_f$ induces a homomorphism $\phi_f: \HAG_{a, b, c} \rightarrow \HAG_{a, b, c}$.
\end{proof}

Now we are ready to prove Proposition \ref{inforapenny}.  Supposing there exists a nontrivial homomorphism from $\HAG$ to $G$ we obtain by Lemma \ref{nicehom} a homomorphism $\phi: \HAG_{a, b, c} \rightarrow G$ such that $\phi([[U_T]])\neq 1_G$ and $\HAG_{b, c}\leq \ker(\phi)$.  Given a subset $\mathcal{S} \subseteq \Sigma$ we define a function $f_{\mathcal{S}}: \Sigma \rightarrow \Sigma\cup \{T\}$ by $f_{\mathcal{S}}(S) = \begin{cases}  S$ if $S\notin \mathcal{S}\\ T$ if $S\in\mathcal{S}\end{cases}$.  The accompanying $\phi_{f_{\mathcal{S}}}: \HAG \rightarrow \HAG$ satisfies 

\begin{center}
$\phi_{f_{\mathcal{S}}}([[U_S]]) = \begin{cases}[[U_S]]$ if $S\notin \mathcal{S}\\ [[U_T]]$ if $S\in \mathcal{S}\end{cases}$
\end{center}

\noindent and so $[[U_S]]\in \ker(\phi\circ\phi_{f_{\mathcal{S}}})$ if and only if $S\notin \mathcal{S}$.  Thus $|\Hom(\HAG, G)| \geq 2^{2^{\aleph_0}}$.

We end this section by restating and proving Theorem \ref{subgroups}:

\begin{B}  The group $\Aut(\HAG)$ contains an isomorphic copy of the full symmetric group $S_{2^{\aleph_0}}$ on a set of cardinality continuum.  Thus $\Aut(\HAG)$ is of cardinality $2^{2^{\aleph_0}}$ and all groups of size at most $2^{\aleph_0}$ are subgroups.
\end{B}

\begin{proof}  Supposing that $\sigma: \Sigma \rightarrow \Sigma$ is a bijection, we get an endomorphism $\phi_{\sigma}: \HAG \rightarrow \HAG$ and notice that $\phi_{\sigma^{-1}}\phi_{\sigma} = \phi_{\sigma}\phi_{\sigma^{-1}} = Id_{\HAG}$.  Thus the mapping $\phi_{\sigma}$ is an automorphism, and it is straightforward to check that $\phi_{\sigma_0}\phi_{\sigma_1} = \phi_{\sigma_0\sigma_1}$, so $\sigma\mapsto \phi_{\sigma}$ is a homomorphism to $\Aut(\HAG)$.  If $\sigma(S)= S' \neq S$ then $\phi_{\sigma}([[U_S]]) = [[U_{S'}]]\neq [[U_S]]$, and so the mapping $\sigma \mapsto \phi_{\sigma}$ has trivial kernel.  Thus we see that $\Aut(\HAG)$ has a copy of $S_{2^{\aleph_0}}$.  From this, and since $\HAG$ is of cardinality $2^{\aleph_0}$, we see that $\Aut(\HAG)$ is of cardinality $2^{2^{\aleph_0}}$.  Since $S_{2^{\aleph_0}}\leq \Aut(\HAG)$ every group of cardinality $\leq 2^{\aleph_0}$ is also a subgroup by considering the left action of a group on itself.

\end{proof}

\end{section}

\begin{section}{Residual slenderness and subgroups}\label{Residualslenderness}

The techniques used so far suggest that residuality can be considered in determining the n-slenderness of small cardinality groups.  Theorem \ref{residually} below makes this explicit.  We explore two subgroups of a group $G$ whose triviality precisely determines n-slenderness in a group of small cardinality.  Further, an example is shown in which these two subgroups are not equal.  From this, a new family of groups is shown to be n-slender.  We motivate and then present Theorem \ref{subgroupofheg} which determines which subgroups of $\HEG$ are n-slender (they are precisely those which do not contain an isomorph of $\HEG$ as a subgroup).  Finally, we leave the reader with two open questions.

\begin{theorem}\label{residually}  If $G$ is a group such that $|G|<2^{\aleph_0}$ then $G$ is n-slender if and only if $G$ is residually n-slender.
\end{theorem}

\begin{proof}  The direction ($\Rightarrow$) is obvious by using the identity homomorphism.  For ($\Leftarrow$), we assume $G$ is residually n-slender and $|G|<2^{\aleph_0}$.  If $G$ fails to be n-slender, we have by Lemma \ref{hagfact2} a nontrivial homomorphism $\phi: \HAG \rightarrow G$.  Letting $g\in \phi(\HAG) \setminus\{1_G\}$ we pick a homomorphism $\psi: G \rightarrow H$ such that $H$ is n-slender and $\psi(g) \neq 1_H$.  Then $\psi\circ \phi$ is a nontrivial homomorphism from $\HAG$ to the slender group $H$, contradicting Lemma \ref{hagfact1}.
\end{proof}

The statement of Theorem \ref{residually} can fail to hold if one drops the condition $|G|<2^{\aleph_0}$.  For example, $\HEG$ is residually n-slender (since $\HEG$ is residually free) but the identity map witnesses that $\HEG$ is not n-slender.  Theorem \ref{residually} motivates the following definition.

\begin{definition}  Given a group $G$ we define the \textit{slender kernel} $\slk(G)$ to be the intersection of all kernels of homomorphisms from $G$ to noncommutatively slender groups, that is

\begin{center}  $\slk(G) = \bigcap \{\ker(\phi)\mid \phi: G \rightarrow H \text{ with } H \text{ n-slender}\}$
\end{center}

\end{definition}

\noindent  This normal subgroup of $G$ records some of the obstruction which exists for the noncommutative slenderness of $G$.  All torsion elements of $G$ are in $\slk(G)$ and all subgroups isomorphic to $\mathbb{Q}$ are included therein.  By Theorem \ref{residually} the slender kernel of $G$ is precisely the obstruction for noncommutative slenderness of $G$ when $|G|<2^{\aleph_0}$ -that is, such a $G$ is noncommutatively slender if and only if $\slk(G)$ is trivial.  By Lemmas \ref{hagfact2} and \ref{hagfact1} we see that the existence of a nontrivial homomorphic image of $\HAG$ in $G$ is also a precise obstruction of n-slenderness for such groups $G$ of small cardinality.  Let $\HAGim(G)$ denote the subgroup

\begin{center} $\HAGim= \langle \bigcup\{\phi(\HAG)\mid \phi: \HAG \rightarrow G \text{ is a homomorphism}\}\rangle$
\end{center}

This subgroup is obviously normal.  By the proof of Theorem \ref{residually} we have $\HAGim(G) \leq \slk(G)$ for any group $G$.  

\begin{example}  We notice that $\HAGim(G)$ need not be the union $\bigcup\{\phi(\HAG)\mid \phi: \HAG \rightarrow G \text{ is a homomorphism}\}$.  To see this we let $G = (\mathbb{Z}/2) * (\mathbb{Z}/2)$.  We know that $\mathbb{Z}/2$ is a homomorphic image of $\HAG$ (this group is not n-slender), and so in this case $\HAGim(G) = G$.  However, any homomorphism $\phi: \HAG \rightarrow  (\mathbb{Z}/2) * (\mathbb{Z}/2)$ has image which is contained in a conjugate of the first or the second copy of $\mathbb{Z}/2$ by applying Lemma \ref{hagfact3}.  Thus if $\langle h_0 \rangle$ is the first copy of $\mathbb{Z}/2$ and $\langle h_1\rangle$ is the second copy, the element $h_0h_1 \in G$ would never be in the image of a homomorphism from $\HAG$.
\end{example}

Since $\HAGim(G) \leq \slk(G)$ it seems natural to ask whether equality always holds.  We give an example of a countabe torsion-free group for which this fails, after first proving a lemma.  Recall that a subgroup $H\leq G$ is \emph{central} provided $H$ is a subgroup of the center of $G$.  Central subgroups are always normal.

\begin{lemma}\label{atomic}  Suppose $\{H_i\}_{i\in I}$ is a collection of groups and $H$ is a group such that for each $i$ we have a monomorphism $\phi_i: H\rightarrow H_i$ with $\phi_i(H)$ in the center of $H_i$.  Let $\ast_{H} H_i$ denote the amalgamated free product obtained by identifying the copies of images of $H$ via the maps $\phi_i$.  If $\phi: \HEG \rightarrow \ast_{H} H_i$, then for some $n\in \mathbb{N}$, $j\in I$ and $g\in \ast_{H} H_i$ we have $\phi(\HEG^n) \leq g^{-1} H_j g$.
\end{lemma}

\begin{proof}  Since $\phi_i(H)$ is central in $H_i$ for every $i\in I$, we have $H \leq \ast_{H} H_i$ a central subgroup, and therefore normal.  The isomorphism $(\ast_{H} H_i)/H \simeq \ast_{i\in I} (H_i/H)$ is clear, and let $\psi:\ast_{H} H_i\rightarrow  \ast_{i\in I} (H_i/H)$ be the quotient map.  Given a map $\phi: \HEG \rightarrow \ast_{H} H_i$, we notice by Lemma \ref{funneling} that for some $n\in \mathbb{N}$, $j\in I$ and $h\in \ast_{H} H_i$ the inclusion $\psi\circ\phi(\HEG^n) \leq h^{-1}(H_j/H)h$.  Selecting $g\in \ast_{H} H_i$ satisfying $\psi(g) = h$ it is easy to see that $\phi(\HEG^n) \leq g^{-1}H_jg$.
\end{proof}

\begin{example}\label{strictinclusion}  Consider $2\mathbb{Z}$ both as a subgroup of $\mathbb{Z}$ as well as a subgroup of $\mathbb{Q}$.  Let $G$ be the amalgamated free product $\mathbb{Z} *_{2\mathbb{Z}}\mathbb{Q}$ which identifies the copy of $2\mathbb{Z}$ in $\mathbb{Z}$ with that in $\mathbb{Q}$.  As both $\mathbb{Z}$ and $\mathbb{Q}$ are torsion-free, the group $G$ is torsion-free.

By Lemma \ref{atomic}, any homomorphism $\phi: \HAG\rightarrow G$ must either have $\phi(\HAG)$ as a subgroup of a conjugate of $\mathbb{Z}$ or a conjugate of $\mathbb{Q}$.  As there is no nontrivial map from $\HAG$ to $\mathbb{Z}$, we see that any nontrivial image of $\HAG$ must lie inside a conjugate of $\mathbb{Q}$.  Thus $\HAGim(G) \leq \langle\langle \mathbb{Q}\rangle\rangle \leq G$, and since each conjugate of $\mathbb{Q}$ is in $\HAGim(G)$ we get $\HAGim(G) = \langle\langle \mathbb{Q}\rangle\rangle$.

Since $\slk(G) \geq \HAGim(G)$ and $\HAGim(G)$ is of index $2$ in $G$, it is clear that $\slk(G) = G$ since n-slender groups are torsion-free.
\end{example}

\begin{theorem}\label{beautiful}  Assume the hypotheses of Lemma \ref{atomic}.  If each of the groups $H_i$ is n-slender, then so is $\ast_{H} H_i$.
\end{theorem}

\begin{proof}  Assume the hypotheses and suppose $\phi: \HEG \rightarrow \ast_{H} H_i$ is a homomorphism.  By Lemma \ref{atomic} select $n\in \mathbb{N}$, $j\in I$ and $g\in \ast_{H} H_i$ so that $\phi(\HEG^n)\leq g^{-1}H_jg$.  Since $H_j$ is slender there exists $m\geq n$ so that $\phi(\HEG^m)$ is trivial.  Thus $\ast_{H} H_i$ is n-slender.
\end{proof}

This theorem provides new examples of n-slender groups, as seen in the next example.

\begin{example}\label{divisible}  Let $\{s_n\}_{n\in \omega}$ be a sequence in $\mathbb{Z} \setminus \{0\}$ such that $s_0 = 1$.  Since $\mathbb{Z}$ is n-slender, the amalgamated free product $\langle\{a_n\}_{n\in \omega} \mid \{a_0=a_n^{s_n}\}_{n\in \omega}\rangle$ is n-slender by Theorem \ref{beautiful}.  If one lets $s_n = n+1$, then the element $a_0$ has an $n$-th root for every $n\in \omega\setminus \{0\}$.  As far as the author is aware, this gives the first known example of an n-slender group which has a nontrivial divisible element.  Notice that a nontrivial group in which every element is divisible cannot be n-slender.  Such a group will either have torsion or it will be torsion-free, in which case one can easily piece together a subgroup which is isomorphic to $\mathbb{Q}$.
\end{example}

We end by motivating a noncommutative version of a theorem regarding slender subgroups and posing two questions.  The following is an immediate corollary to Theorem \ref{residually}:

\begin{corollary}\label{subgroupsofHEG}  Every subgroup of the Hawaiian earring group of cardinality $<2^{\aleph_0}$ is noncommutatively slender.
\end{corollary}

Although the group $\HEG$ is locally free, there exist countable subgroups which are not free (see the discussion following \cite[Theorem 6]{Hi}).  Thus Corollary \ref{subgroupsofHEG} gives examples of non-free n-slender groups.  By contrast, the group $\prod_{\omega}\mathbb{Z}$ is \emph{$\aleph_1$-free}- that is, every countable subgroup is free abelian \cite{Sp}.  There is an analogous abelian version of Corollary \ref{subgroupsofHEG}, which follows immediately from the classification of slender groups of small cardinality in \cite{Sa}:

\begin{obs}\label{abelianresidual}  Every subgroup of $\prod_{\omega}\mathbb{Z}$ of cardinality $<2^{\aleph_0}$ is slender.
\end{obs}

In light of $\aleph_1$-freeness and the fact that free abelian groups are slender, this observation does not furnish many new examples unless the continuum hypothesis fails.  Moreover, in light of \cite{Nu} one gets the stronger observation:

\begin{obs}\label{nunke}  A subgroup of $\prod_{\omega}\mathbb{Z}$ fails to be slender if and only if it contains a subgroup isomorphic to $\prod_{\omega}\mathbb{Z}$.
\end{obs}

We give the non-abelian analog to this stronger observation.

\begin{theorem}\label{subgroupofheg}  A subgroup of $\HEG$ fails to be n-slender if and only if it contains a subgroup isomorphic to $\HEG$.
\end{theorem}

For this we first prove a straightforward lemma regarding free groups.

\begin{lemma}\label{obvious}  Let $F(X)$ be the free group on generators $X$, let $Y \subseteq X$ and $w\in F(X)$ be a reduced word that utilizes an element of $X \setminus Y$.  Letting $t$ be a symbol such that $t\notin X$, the mapping $f: \{t\}\cup Y \rightarrow \{w\}\cup Y$ given by $f(t) = w$ and $f(y) = y$ for $y\in Y$ extends to an isomorphism $\phi: F(\{t\}\cup Y) \rightarrow \langle \{w\}\cup Y \rangle$.
\end{lemma}

\begin{proof}  By freeness of $F(\{t\}\cup Y)$ we get an extension $\phi: F(\{t\}\cup Y) \rightarrow  \langle \{w\}\cup Y \rangle$, and we need only check that this is injective.  We consider elements of free groups as reduced words in the appointed generators.  Write $w \equiv w_0w_1w_2w_1^{-1}w_3$, where $w_0$ is the maximal prefix of $w$ which uses only letters in $Y^{\pm 1}$, $w_3$ is the maximal suffix of $w$ which uses only letters in $Y^{\pm 1}$, and $w_2$ is the cyclic reduction of the remaining middle word $w_1w_2w_1^{-1}$.  Since $w$ uses a letter not in $Y$, the remaining middle word $w_1w_2w_1^{-1}$ is not the empty word, and therefore $w_2$ is also nonempty.

Supposing we have a multiplication $wvw$ with $v$ nontrivial $F(Y)$, we consider the maximum extent of letter cancellation.  We have

\begin{center}
$wvw \equiv  w_0w_1w_2w_1^{-1}w_3 v w_0w_1w_2w_1^{-1}w_3$
\end{center}

\noindent and the greatest extent of letter cancellation has $w_3 v w_0$ cancelling entirely, so that $w_1^{-1}w_3 v w_0w_1$ cancels entirely, but the cancellation may go no further since $w_2$ was cyclically reduced.  Supposing we have multiplication $wvw^{-1}$ with $v$ nontrivial in $F(Y)$, we consider the maximum extent of cancellation within

\begin{center}  $wvw^{-1}\equiv w_0w_1w_2w_1^{-1}w_3 vw_3^{-1}w_1w_2^{-1}w_1^{-1}w_0^{-1}$
\end{center}

\noindent  Since $v$ is nontrivial we see that $w_3 vw_3^{-1}$ is not trivial.  Thus in both cases, the nontrivial word $w_2^{\pm 1}$ remains untouched after a maximal cancellation, and the multiplication $w^{-1}v w$ yields the same conclusion, as do products $ww$ and $w^{-1}w^{-1}$.  This is sufficient for showing that a nontrivial reduced word in $F(\{t\}\cup Y)$ maps nontrivially.
\end{proof}

\begin{proof}(of Theorem \ref{subgroupofheg})  Only the $(\Rightarrow)$ direction is nontrivial.  Suppose $G \leq \HEG$ is not n-slender and let $\phi_0: \HEG \rightarrow G$ witness this.  By Lemma \ref{continuoustoconjugation} we can conjugate both $G$ and the homomorphism and obtain a new homomorphism $\phi_1$ determined by a mapping $a_n \mapsto W_n$ where $[W_n] \in \HEG^{j_n}$ and $j_n \nearrow \infty$ which witnesses that the conjugate of $G$ is not n-slender.  It will be sufficient to find a subgroup of this conjugate which contains $\HEG$ as a subgroup, so without loss of generality we replace $G$ with this conjugate.

Since $\phi_1$ witnesses the negation of n-slenderness, we can select a sequence $V_n$ of reduced words such that $[V_n]\in \HEG^n$ and $\phi_1(V_n)\in \HEG^n \setminus \{1\}$ and we let $U_n$ be the reduced word such that $U_n \in \phi(V_n)$.  The mapping $f: a_n \mapsto U_n$ induces an endomorphism $\phi_2:\HEG \rightarrow\HEG$ such that $\phi_2(\HEG) \leq G$ and $\phi_2([a_n]) \in \HEG^n\setminus\{1\}$ for all $n\in \omega$.  Now, for each $n\in \omega$ there exists $m_n\in \omega$ such that $p_{m_n}(\phi_2([a_n]))\neq 1$.  Thus by using a subsequence we produce an endomorphism $\phi$ such that $\phi(\HEG) \leq G$, $p_{m_n}(\phi([a_n])) \neq 1$ and $\phi([a_{n+1}])\in \HEG^{m_n}$.  This last condition guarantees that $p_{m_{n-1}} \circ \phi \circ p_n = p_{m_{n-1}} \circ \phi$.  It also guarantees that given $n < j$ there exists a letter utilized in $p_{m_j}\circ\phi([a_n])$ which is not one of the generators of the free group  $\HEG_{m_n} \cap\HEG^{m_j}$.

We claim that for each $n\in \omega$ the restriction $p_{m_{n-1}}\circ\phi\upharpoonright \HEG_n$ is an injection.  Fix $n$.  We know $p_{m_{n-1}}\circ \phi([a_0])$ utilizes a letter that is not a generator in $\HEG_{m_{n-1}} \cap \HEG^{m_0}$, so by Lemma \ref{obvious} we get the isomorphism $$\langle t_0\rangle * (\HEG_{m_{n-1}} \cap \HEG^{m_0}) \simeq \langle \{p_{m_{n-1}}\circ \phi([a_0])\} \cup (\HEG_{m_{n-1}} \cap \HEG^{m_0})\rangle$$  Since the subgroup $\langle p_{m_{n-1}}\circ \phi([a_1]),\ldots,   p_{m_{n-1}}\circ \phi([a_{n-1}])\rangle$ is included in the group $\HEG_{m_{n-1}} \cap \HEG^{m_0}$, we get \emph{a fiortiori} that $$\langle t_0\rangle * \langle p_{m_{n-1}}\circ \phi([a_1]),\ldots,   p_{m_{n-1}}\circ \phi([a_{n-1}])\rangle \simeq \langle p_{m_{n-1}}\circ \phi([a_0]),\ldots,   p_{m_{n-1}}\circ \phi([a_{n-1}])\rangle$$  Continuing to argue in this manner we get that $\langle p_{m_{n-1}}\circ \phi([a_0]),\ldots,   p_{m_{n-1}}\circ \phi([a_{n-1}])\rangle$ is a free group in its listed generators, so $p_{m_{n-1}}\circ\phi\upharpoonright \HEG_n$ is injective.

Now, given $[W]\in \HEG \setminus\{1\}$ we select $n\in \omega$ such that $p_n([W])\neq 1$.  Then $p_{m_{n-1}} \circ \phi \circ p_n([W])\neq 1$ since $p_{m_{n-1}}\circ\phi\upharpoonright \HEG_n$ is injective.  Thus $$1\neq  p_{m_{n-1}} \circ \phi \circ p_n([W])= p_{m_{n-1}} \circ \phi([W])$$ so that $\phi([W])\neq 1$.  Then $\phi$ is a monomorphism and we are done.
\end{proof}

\noindent The classification in Theorem \ref{subgroupofheg} cannot be strengthened by more generally considering subgroups of residually free groups, since $\prod_{\omega}\mathbb{Z}$ is residually free, fails to be n-slender, and does not contain $\HEG$ as a subgroup.

In light of the result of \cite{Nu} we ask whether the analogous situation holds in the non-abelian case:

\begin{question}\label{big}  Does there exist a countable set of groups $\{G_n\}_{n\in \omega}$ such that a group fails to be n-slender if and only if it includes one of the $G_n$ as a subgroup?
\end{question}

\noindent As a weakening of Question \ref{big}, we ask whether the main result of \cite{Sa} holds in the non-abelian case:

\begin{question}\label{small}  Does there exist a countable set of groups $\{G_n\}_{n\in \omega}$ such that a countable group fails to be n-slender if and only if it includes one of the $G_n$ as a subgroup?
\end{question}

\noindent As has already been mentioned, it is even unknown whether there exists a countable group not containing $\mathbb{Q}$ or torsion which fails to be n-slender.

\end{section}


\begin{thebibliography}{abcdefghijk}

\bibitem[ADTW]{ADTW} S. Akiyama, G. Dorfer, J. Thuswaldner, R. Winkler, \emph{On the fundamental group of the Sierpi\'{n}ski-gasket}, Topol. Appl 156 (2009), 1655-1672.

\bibitem[BS]{BS} W. Bogley, A. Sieradski, \emph{Weighted combinatorial group theory and wild metric complexes}, Groups--Korea '98 (Pusan), de Gruyter, Berlin, 2000, 53-80.

\bibitem[CC1]{CC1} J. Cannon, G. Conner, \emph{The combinatorial structure of the Hawaiian earring group}, Topol. Appl 106 (2000), 225-271.

\bibitem[CC2]{CC2} J. Cannon, G. Conner, \emph{On the fundamental groups of
one-dimensional spaces}, Topol. Appl 153 (2006), 2648-2672.

\bibitem[C]{C} G. Conner, personal communication.

\bibitem[CoCo]{CoCo}  G. Conner, S. Corson, \emph{A note on automatic continuity}, arXiv 1710.04787

\bibitem[CHM]{CHM} G. Conner, W. Hojka, M. Meilstrup, \emph{Archipelago groups}, Proc. of the Amer. Math. Soc. 143 (2015), 4973-4988.

\bibitem[CS]{CS}  G. R. Conner, K. Spencer. \emph{Anomalous behavior of the Hawaiian earring group}, J. Group Theory 8 (2005), 225–-227.

\bibitem[Co1]{Co1} S. Corson, \emph{Torsion-free word hyperbolic groups are n-slender}, Int. J. Algebra Comput. 26 (2016), 1467-1482.

\bibitem[Co2]{Co2} S. Corson, \emph{Root extraction in one-relator groups and slenderness}, arXiv:1709.02906

\bibitem[D]{D} R. Dudley, \emph{Continuity of homomorphisms}, Duke Math. J. 28 (1961), 587-594.

\bibitem[E1]{E1} K. Eda,  \emph{Free $\sigma$-products and noncommutatively slender groups}, J. Algebra 148 (1992), 243-263.

\bibitem[E2]{E2} K. Eda, \emph{Free $\sigma$-products and fundamental groups of subspaces of the plane}, Topol. Appl 84 (1998), 283-306.

\bibitem[E3]{E3} K. Eda, \emph{Atomic property of the fundamental groups of the Hawaiian earring and wild locally path-connected spaces}, J. Math. Soc. Japan 63 (2011), 769-787.

\bibitem[EK]{EK} K. Eda, K. Kawamura, \emph{The singular homology of the Hawaiian earring},  J. London Math. Soc. 62 (2000), 305-310.

\bibitem[F]{F} L. Fuchs. Infinite Abelian Groups, Vols 1,2, Academic Press, San Diego 1970, 1973.

\bibitem[FZ]{FZ} H. Fischer, A. Zastrow, \emph{The fundamental groups of subsets of closed surfaces inject into their first shape groups}, Algebr. Geom. Topol. 5 (2005), 1655-1676.

\bibitem[Hi]{Hi} G. Higman,  \emph{Unrestricted free products and varieties of topological groups},  J. London Math. Soc. 27 (1952), 73-81.

\bibitem[Ho]{Ho} W. Hojka, \emph{The harmonic archipelago as a universal locally free group}, J. Algebra 437 (2015), 44-51.

\bibitem[Ku]{Ku} K. Kunen.  Set Theory:  An Introduction to Independence Proofs, North-Holland, Amsterdam, 1980.

\bibitem[MM]{MM} J. Morgan, I Morrison, \emph{A van Kampen theorem for weak joins}, Proc. London Math. Soc. 53 (1986), 562-576.

\bibitem[Nu]{Nu} R. Nunke, \emph{Slender groups}, Bull. Amer. Math. Soc. 67 (1961), 274-275.

\bibitem[Sa]{Sa} E. S\c asiada, \emph{Proof that every countable and reduced torsion-free abelian group is slender}, Bull. Acad. Polon. Sci. 7 (1959), 143-144.

\bibitem[Sm]{Sm} B. de Smit, \emph{The fundamental group of the Hawaiian earring is not free}, Int. J. Algebra Comput. 2 (1992), 33-37.

\bibitem[Sp]{Sp} E. Specker, \emph{Additive Gruppen von folgen Ganzer Zahlen}, Portugal. Math. 9 (1950), 131-140.

\bibitem[T]{T} T. Tlas,  \emph{Big free groups are almost free}, Int. J. Algebra Comput. 25 (2015), 855-.

\bibitem[Z]{Z} A. Zastrow, \emph{The non-abelian Specker group is free}, J. Algebra 229 (2000), 55-85.

\end{thebibliography}
\end{document}